\documentclass{amsart}
\font\tenrm = cmr20 at 12pt
\tenrm
\usepackage{latexsym}
\usepackage{graphicx}
\usepackage{amsfonts,amsmath,amsthm,amssymb,enumerate}
\usepackage{pb-diagram,pb-xy}
\usepackage{subfigure}
\usepackage[all]{xy}
\usepackage{mathrsfs}
\usepackage{graphics}
\usepackage{epsfig}
\usepackage{psfrag}
\usepackage{amsmath}
\usepackage{amssymb}
\usepackage{mathrsfs}
\usepackage{epstopdf}

\usepackage[left=1.5in,right=1.5in,top=1.5in,bottom=1.5in]{geometry}

\usepackage[colorlinks=true,
            linkcolor=black,
            urlcolor=blue,
            citecolor=blue]{hyperref}

\usepackage[usenames,dvipsnames]{color}
\DeclareMathOperator{\semi}{\centering{\substack{<\\ (=)}}}

\DeclareMathOperator{\Pic}{Pic}

\newtheorem*{corollary*}{Corollary}
\newtheorem*{theorem*}{Theorem}
\newtheorem{theorem}{Theorem}[section]
\newtheorem{lemma}[theorem]{Lemma}
\newtheorem{prop}[theorem]{Proposition}

\newtheorem{corollary}[theorem]{Corollary}

\theoremstyle{definition}
\newtheorem{definition}[theorem]{Definition}
\newtheorem{example}{Example}[theorem]

\theoremstyle{remark}
\newtheorem{rem}{Remark}[theorem]

\newcommand{\bb}{\mathbb}

\numberwithin{equation}{section}
\newcommand{\Stab}{\rm Stab}

\def\Z{\mathbb{Z}}
\def\R{\mathbb{R}}
\def\Q{\mathbb{Q}}
\def\C{\mathbb{C}}
\def\P{\mathbb{P}}

\title[Change of Polarization]{Change of polarization for moduli of sheaves on surfaces as Bridgeland wall-crossing}
\author{Aaron Bertram}
\address{Department of Mathematics\\ University of Utah\\
155 S 1300 E \\
Salt Lake City, UT 84112}
\email{bertram@math.utah.edu}

\author{Cristian Martinez }
\address{Department of Mathematics\\ University of Utah\\
155 S 1300 E \\
Salt Lake City, UT 84112}
\email{martinez@math.utah.edu}

\subjclass[2000]{}
\begin{document}
\maketitle

\begin{abstract} 
We prove that the ``Thaddeus flips" of $L$-twisted sheaves appearing in \cite{MW} can be obtained via Bridgeland wall-crossing. Similarly, we realize the change of polarization for moduli spaces of 1-dimensional Gieseker semistable sheaves on a surface by varying a family of stability conditions.
\end{abstract}

\tableofcontents

\section{Introduction}
The notion of stability for torsion-free sheaves on a smooth projective complex surface $X$ depends on the choice of a divisor class $H\in \mbox{Amp}(X)$ in the ample cone of the surface. The coarse moduli spaces $M_H(v)$ of $H$-Gieseker semistable sheaves on $X$ with Chern character $v$ are projective and can be constructed via Geometric Invariant Theory (GIT) (\cite{Gies77}). There is a wall and chamber decomposition of the ample cone of the surface $\mbox{Amp}(X)$ such that $M_H(v)$ and $M_{H'}(v)$ are isomorphic when $H$ and $H'$ belong to the same chamber. In the 90s there was a great deal of interest in studying how these moduli spaces relate to each other for polarizations in different chambers. Results obtained independently by Ellingsrud and G\"{o}ttsche \cite{EG} and Friedman and Qin \cite{FQ} for rank-two sheaves, and by Matsuki and Wentworth \cite{MW} in arbitrary rank show that when crossing a wall in $\mbox{Amp}(X)$, the moduli space $M_H(v)$ goes through a sequence of ``Thaddeus flips''  of moduli spaces of {\it twisted} sheaves. 

 If $L$ is a $\Q$-line bundle on $X$,  then a torsion-free sheaf $E$ is $L$-twisted $H$-Gieseker semistable \cite[Definition 3.2]{MW} if for all subsheaves $A\hookrightarrow E$ one has:
\begin{equation}\label{eqntwisted}
\left(\frac{\chi(A\otimes L)}{r(A)}-\frac{\chi(E\otimes L)}{r(E)}\right)+t(\mu_{H}(A)-\mu_H(E))\leq 0 \  \ \text{for}\ \ t\gg 0,
\end{equation}
where the Euler characteristic $\chi(\underline{\hspace{0.4cm}}\otimes L)$ is defined formally via the Riemann--Roch theorem\footnote{The definition of twisted semistability in \cite{MW} is given in terms of reduced Hilbert polynomials. One needs the Riemann--Roch theorem to see the equivalence with \eqref{eqntwisted}.}. Coarse moduli spaces of $L$-twisted $H$-Gieseker semistable sheaves were constructed in \cite{MW} and proven to be projective.

The introduction of Bridgeland stability conditions \cite{Bstab} provides us with new tools to study the birational geometry of $M_H(v)$. For instance, it has been shown that running a
directed Minimal Model Program (MMP) for $M_H(v)$ when $X$ is $K3$  \cite{BMK3}, abelian \cite{YosAb}, Enriques \cite{Nuer}, or the projective plane \cite{ABCH,BMW,CHW}, corresponds to varying a stability condition on $X$. Also, when $v=[\C_x]$ the moduli space $M_H(v)$ is canonically isomorphic to $X$ and Toda \cite{TodaPS} has proven that any MMP for $X$ can be obtained by varying stability conditions. 

By the boundedness results for walls on certain rays in $\mbox{Stab}(X)$ obtained by Lo and Qin \cite{LQ} and generalized by Maciocia \cite{MACIO}, we know that moduli spaces of $L$-twisted $H$-Gieseker semistable sheaves are moduli spaces of Bridgeland semistable objects. Then one can ask if the variation of GIT obtained by Matsuki and Wentworth relating moduli spaces of Gieseker semistable sheaves for different polarizations can be interpreted as Bridgeland wall-crossings, and moreover if we can find a specific family of stability conditions realizing this variation.

For a fixed Chern character cohomology ``vector'' $v=(r(v),c_1(v),ch_2(v))$ on $X$, the vectors
\begin{equation}\label{definition_d_t}
 \alpha_t=\left(1,-\frac{K_X}{2}+L+tH,d_t\right)\ \  \text{with}\ \ d_t=-\frac{\chi(v)}{r(v)}-\frac{c_1(v)}{r(v)}\left(L+tH\right)+\chi(\mathcal{O})
 \end{equation}
are perpendicular to $v$ in cohomology, i.e.:
$$
\langle v, \alpha_t \rangle = \int_X ch_2(v) + c_1(v) \cdot  \left(-\frac{K_X}{2}+L+tH\right) + r(v) d_t = 0,
$$ 
where $\langle\ \cdot\ ,\ \cdot\ \rangle$ denotes the Poincar\'e pairing on cohomology classes.
 
If $E$ is a torsion-free sheaf with $ch(E)=v$ and $A$ is torsion-free then
\begin{align*}
\langle ch(A), \alpha_t \rangle&=\chi(A)+c_1(A)(L+tH)+r(A)d_t-r(A)\chi(\mathcal{O})\\
&= r(A)\left(\frac{\chi(A)}{r(A)}+\frac{c_1(A)}{r(A)}(L+tH)+d_t-\chi(\mathcal{O})\right)\\
&=r(A)\left(\frac{\chi(A\otimes L)}{r(A)}-\frac{\chi(E\otimes L)}{r(E)}+t(\mu_{H}(A)-\mu_H(E))\right).
\end{align*}

Thus $E$ is $L$-twisted $H$-Gieseker semistable if and only if for every subsheaf $A\hookrightarrow E$ one has
\begin{equation}\label{gieseker_as_a_pairing}
\langle ch(A), \alpha_t \rangle\leq 0 \ \ \text{for}\ \ t\gg 0.
\end{equation}

This is the key observation that will allow us to study the change of polarization. Each of the orthogonal classes $\alpha_t\in v^{\perp}$ produces a line bundle $\mathcal{L}_t\in \mbox{Pic}(M_H(v))_{\R}$, via the determinant line bundle construction, that is ample for $t$ large enough. On the other hand, Bayer and Macri \cite{BM} associate to every stability condition $\sigma\in \mbox{Stab}(X)$, a nef divisor class $\mathcal{L}_{\sigma}$ on the stack $\mathcal{M}_{\sigma}(v)$ of flat families of $\sigma$-semistable objects of Chern character $v$. Our idea is to study a particular ray of stability conditions $\{\sigma_t\}\subset\mbox{Stab}(X)$, associated to the classes $\alpha_t$, for which $\mathcal{L}_t=\mathcal{L}_{\sigma_t}$ whenever $\sigma_t$-semistability coincides with $H$-Gieseker semistablility. However, since the condition \eqref{gieseker_as_a_pairing} is only asymptotic we will need to establish it for specific values of $t$; that is, bound the values of $t$ that correspond to walls on the ray $\{\sigma_t\}$.  

To see the variation on the moduli spaces of Gieseker semistable sheaves as Bridgeland wall-crossings, we need to allow $H$ to move in the ample cone. Ideally, one would consider the family of stability conditions associated to the vectors 
$$
 \alpha_{t,H}=\left(1,-\frac{K_X}{2}+tH,d_t\right),\ \text{for arbitrary}\ H\in\mbox{Amp}(X),
$$
but in the absence of a global boundedness result for Bridgeland walls, one needs to restrict to finitely generated convex cones inside $\mbox{Amp}(X)$. In this case, we can use the boundedness result of Maciocia \cite{MACIO} and a detailed study of the walls to find convex regions in every Gieseker chamber, for which the line bundles $\mathcal{L}_{t,H}$ (associated to $\alpha_{t,H}$) are ample.

The main result of this note is:
\begin{theorem*}[Theorem \ref{main2}] Let $H'$ and $H''$ be ample classes in adjacent chambers in the wall and chamber decomposition of $\mbox{Amp}(X)$ for the class $v$. Then there is a one dimensional family of stability conditions $\{\gamma_s\}_{s\in(-1,1)}$ and rational numbers $-1=s_0<s_1<\cdots <s_n=1$ such that 
each moduli space $M_{\gamma_s}(v)$ of Bridgeland semistable objects is a moduli space of twisted sheaves for every $s$ and is constant on $(s_i,s_{i+1})$, it equals $M_{H'}(v)$ for $s\in(s_0,s_1)$ and equals $M_{H''}(v)$ for $s\in(s_{n-1},s_n)$.
\end{theorem*}

Thus every wall in $\mbox{Amp}(X)$ corresponds to a finite sequence of Bridgeland walls, and it will follow from the proof that we get a Bridgeland wall for each Thaddeus flip obtained in \cite{MW}. Applying a similar analysis, we will also find a one-parameter family of stability conditions explaining the change of polarization for 1-dimensional Gieseker semistable sheaves; in this case, there is a one-to-one correspondence between Bridgeland walls and Gieseker walls.

The organization of the paper is as follows. In Section \ref{definitions} we set up the notation for the type of stability conditions that we will use, the main reference is \cite{bertram}. For the basics on stability conditions the reader may consult \cite{Bstab} or the expository notes \cite{HSTAB}. In Section \ref{boundedness} we prove that walls intersecting certain rays in $\mbox{Stab}(X)$ are bounded above. In Section \ref{changeofpolarization} we prove Theorem \ref{main2}. In Section \ref{change_for_1-dim} we treat the case of change of polarization for 1-dimensional sheaves. Finally, in Section \ref{birational} we use our methods to give a proof of the result of Toda \cite{TodaPS} that a contraction of a $-1$-curve can be interpreted as a Bridgeland wall-crossing. 

While preparing this manuscript we were informed by Kota Yoshioka about his work on perverse sheaves \cite{Yos14}. He also obtains Theorem \ref{main2}.   

\section{Preliminaries}\label{definitions}

Let $X$ be a smooth projective complex surface and $H\in \mbox{Amp}(X)$ a polarization. A torsion-free sheaf $\mathcal{F}$ is slope semistable with respect to $H$ (or $H$-semistable) if for all subsheaves $A\hookrightarrow \mathcal{F}$ one has $\mu_H(A)\leq \mu_H(\mathcal{F})$ where $\mu_H(\ \cdot\ )=c_1(\ \cdot\ )\cdot H/r(\ \cdot\ )$. Every torsion-free sheaf $E\in\mbox{Coh}(X)$ has a unique filtration
$$
0=E_0\subset E_1\subset \cdots E_{n-1}\subset E_n=E 
$$ 
such that every factor $F_i=E_i/E_{i-1}$ is $H$-semistable and $\mu_H(F_1)>\mu_H(F_2)>\cdots >\mu_H(F_n)$. We refer to $F_1$ and $F_n$ as the first and the last $H$-semistable factors respectively.

For a fixed Chern character $v=(r(v),c_1(v),ch_2(v))$ on $X$, the set of hyperplanes 
$$
\mathcal{H}_A=\{H\in \mbox{Amp}(X)_{\Q}\colon \mu_H(A)=\mu_H(v)\},
$$
for which there is an injective map $A\hookrightarrow E$ to some slope semistable sheaf $E$ of type $v$ and  
$$
\frac{c_1(A)}{r(A)}\neq \frac{c_1(E)}{r(E)},
$$ 
is locally finite in $N^1(X)_{\Q}$. Moreover, if $\Delta$ is a finitely generated convex cone in $\mbox{Amp}(X)_{\Q}$ then only finitely many of these hyperplanes intersect $\Delta\setminus\{0\}$ (\cite{MW}). We refer to these hyperplanes as walls for slope stability; every connected component of $\mbox{Amp}(X)_{\Q}\setminus \bigcup_{A}\mathcal{H}_A$ is called a chamber. A wall $\mathcal{H}_A$ is a wall for Gieseker stability for the class $v=ch(E)$ if also 
$$
\frac{\chi(A)}{r(A)}\geq \frac{\chi(E)}{r(E)}.
$$     

The wall and chamber decomposition for Gieseker stability is not well behaved since, in general, a torsion-free sheaf can pass from being stable to being unstable without ever being strictly semistable. For an example of this phenomenon see Example \ref{p1p1example}.

Denote by $\widetilde{\mbox{NS}}(X)$ the extended Neron-Severi group $H^0(X,\Z)\oplus \mbox{NS}(X)\oplus H^4(X,\Z)$.  The vector space $\widetilde{\mbox{NS}}(X)_{\R}=\widetilde{\mbox{NS}}(X)\otimes \R$ carries a nondegenerate Poincar\'e pairing
$$
\langle (u_0,u_1,u_2),(v_0,v_1,v_2)\rangle:=u_0v_2+u_1v_1+u_2v_0.
$$
By the Riemann--Roch theorem we obtain
$$
\langle u,v\rangle=\chi(u\cdot \mbox{td}(X)^{-1}\cdot v).
$$ 

As observed by the first author in \cite{bertram}, a vector $\alpha=(\alpha_0,\alpha_1,\alpha_2)\in \widetilde{\mbox{NS}}(X)_{\R}$ with $\alpha_0>0$ and satisfying the Bogomolov inequality $\alpha_1^2>2\alpha_0\alpha_2$, induces a stability condition $\sigma_{\alpha}^H=(Z_{\alpha}^H,\mathcal{A}_{\alpha}^H)\in \mbox{Stab}(X)$ for every ample class $H\in N^1(X)$. The heart $\mathcal{A}_{\alpha}^H$ is obtained in the usual way by tilting with respect to the torsion pair
\begin{align*}
\mathcal{Q}_{\alpha}^{H}&=\{Q\in \mbox{Coh}(X)\colon Q\ \text{is torsion, or}\ \langle ch(B),\alpha H\rangle>0\ \ \text{for all quotients}\ \ Q/\mbox{tor}(Q) \twoheadrightarrow B\},\\
\mathcal{F}_{\alpha}^H&=\{F\in \mbox{Coh}(X)\colon F\ \text{is torsion-free, and}\ \langle ch(A),\alpha H\rangle \leq 0\ \ \text{for all subsheaves}\ \ A\hookrightarrow F \}.
\end{align*}

The central charge is given by
\begin{equation}\label{central_charge}
Z_{\alpha}^H(E)= - \langle ch(E),\alpha \rangle+\sqrt{-1}\langle ch(E),\alpha\cdot H\rangle,
\end{equation}
where $\alpha \cdot H$ is the product in cohomology. 
\section{A boundedness result}\label{boundedness}

Fix a Chern character vector $v=(c_0,c_1,c_2)\in \widetilde{\mbox{NS}}(X)_{\Q}$ with $c_0> 0$, and an ample class $H\in N^1(X)_{\Q}$. Consider the vector 
$$
\alpha_t=(1,D_t,d_t),
$$
where $D_t=-\frac{K_X}{2}+tH$ and $d_t\in H^4(X,\Q)$ is chosen as in \eqref{definition_d_t} to guarantee that $\langle v,\alpha_t\rangle=0$. 

For an ample class $H'\in N^1(X)_{\Q}$ (not necessarily equal to $H$) and $t$ such that $D_t^2>2d_t$, we obtain a stability condition $\sigma_t^{H'}:=\sigma_{\alpha_t}^{H'}$. An object $E\in \mathcal{A}_t^{H'}$ with $ch(E)=v$ is $\sigma_t^{H'}$-semistable if and only if for all sub-objects $A\hookrightarrow E$ in $\mathcal{A}_t^{H'}$ one has
\begin{equation}\label{bridgelandcondition}
\frac{\langle ch(A),\alpha_{t}\rangle}{\langle ch(A), \alpha_t H'\rangle}\leq \frac{\langle v, \alpha_t\rangle}{\langle v, \alpha_t H'\rangle}\ \ \Leftrightarrow\ \  \langle ch(A),\alpha_{t}\rangle\leq 0.
\end{equation}

A \textbf{sheaf} $\mathcal{F}$ is in $\mathcal{A}^{H'}_t$ if and only if its last $H'$-semistable factor $F_n$ satisfies
$$
\mu_{H'}(F_n)>\frac{K_X H'}{2}-tHH'.
$$ 
Hence, it follows that the categories $\mathcal{A}^{H'}_t$ ``approach'' $\mbox{Coh}(X)$ as $t$ approaches infinity. 

Since condition \eqref{bridgelandcondition} is equivalent to the Gieseker condition for $t\gg 0$, then intuitively $\sigma^{H'}_t$-stability coincides with $H$-Gieseker stability for $t\gg 0$. This is a version of Bridgeland's large volume limit \cite[Proposition 14.2]{B2}. A proof in our coordinates appears in \cite{bertram} and we will sketch it below for convenience of the reader. It is remarkable that this limit result is independent of the class $H'$. This is one advantage of using our coordinates.

In order to produce a family of stability conditions realizing the change of polarization for the Gieseker moduli spaces, we need to find stability conditions whose only semistable objects are precisely the Gieseker semistable sheaves. We need to remove the asymptotic condition on $t$. This is equivalent to proving that the values of $t$, corresponding to intersections of walls for the class $v$ in $\mbox{Stab}(X)$ with the ``ray'' $\{\sigma^{H'}_t\}_t$, are bounded above. Assume without loss of generality that $(H')^2=HH'$. We will prove the following: 

\begin{theorem}\label{main}
Fix $L\in \Pic(X)_{\Q}$. Consider the vector $\alpha_t=(1,L+D_t,d_t)$, where as before $d_t$ is chosen such that $\langle v,\alpha_t\rangle=0$. Assume that $H$ and $H'$ are in the same chamber of $\mbox{Amp}(X)_{\Q}$ for $L$-twisted $H$-Gieseker stability. Then there exists $t_0>0$ such that for all $t>t_0$, an object $E$ of class $v$ is $\sigma_t^{H'}$-stable precisely if it is an $L$-twisted $H$-Gieseker stable sheaf.
\end{theorem}

We will postpone the proof of Theorem \ref{main} to the end of this section. 

Continuing with the discussion, if an object $E$ of class $v$ is $\sigma^{H'}_t$-semistable for $t\gg 0$, then $E$ must be a sheaf since otherwise $\mathcal{H}^{-1}(E)[1]$ would destabilize $E$ for large values of $t$. Since the stability condition \eqref{bridgelandcondition} (for the class $v$) is equivalent to the $L$-twisted $H$-Gieseker condition, then every sheaf that is $\sigma^{H'}_t$-semistable for all $t$ large enough should be $L$-twisted $H$-Gieseker semistable. Conversely, if $E\in \mathcal{A}_t^{H'}$ is an $L$-twisted semistable sheaf with $ch(E)=v$ then no subobject $A\hookrightarrow E$ can destabilize $E$ for sufficiently large $t$. Indeed, if $A$ is a subobject of $E$ for $t\gg 0$ then $A$ must be a subsheaf and therefore can not destabilize $E$ since 
$$
\langle ch(A),\alpha_t\rangle= r(A)\left(\frac{\chi(A\otimes L)}{r(A)}-\frac{\chi(E\otimes L)}{r(E)}\right)+ r(A)\left(\mu_{H}(A)-\mu_{H}(E)\right)t.
$$   

When $H=H'$, the existence of $t_0$ in Theorem \ref{main} follows from a boundedness result of Maciocia. The standard coordinates for Bridgeland stability conditions introduced by Bridgeland for K3 surfaces in \cite{B2} and generalized by Arcara and the first author in \cite{AB} depend on two numerical classes $\beta, H\in N^1(X)_{\Q}$ with $H$ ample and are given by the vectors
$$
\alpha_{\beta, tH} =\left(1, -\beta, \frac{\beta^2}{2}-\frac{t^2H^2}{2}\right),
$$ 
that clearly satisfy the Bogomolov inequality for every $t>0$. This choice of vectors is natural from the physics point of view since the corresponding central charge takes the form
$$
\int_X e^{-\beta-\sqrt{-1}tH}v=-\langle v, \alpha_{\beta, tH}\rangle+\sqrt{-1}\langle v, \alpha_{\beta, tH}H\rangle.
$$

\begin{theorem}[\cite{MACIO}]\label{maciocia}
Write $\beta=x_0H+u_0G$ for some divisor $G$ with $G\cdot H=0$ and $G^2=-1$. For $x\in \R$, let $\beta_x=xH+u_0G$. Then, in the half plane 
$$
\Pi_{u_0}=\{\sigma^H_{\alpha_{x,y}}\colon \alpha_{x,y}=(1,-\beta_x,\frac{1}{2}(\beta_x^2-y^2H^2)),\ \ x\in\R,\ y>0\}\subset \Stab(X)
$$
the walls are semicircles of bounded center and radius.
\end{theorem}

In particular, if $-\beta=-\frac{K_X}{2}+L+tH$ and $H'=H$ then the ray $\{\sigma^{H'}_t\}_{t}$ embeds as a one-parameter family in one of Maciocia's half planes. 

We remark that the standard techniques used in \cite{LQ} to show that walls are bounded do not work for us since the categories $\mathcal{A}^{H'}_t$ are changing for every $t$ and therefore the set of objects destabilizing an $L$-twisted $H$-Gieseker semistable sheaf for some $t>t_0$ is not necessarily bounded. Our strategy is to extend Maciocia's result to the case when $H$ and $H'$ are in the same chamber for slope stability, it will then follow that $H$ and $H'$ can be chosen in the same chamber for Gieseker stability. We start with the following:

\begin{lemma}\label{important}
If an $L$-twisted $H$-Gieseker semistable sheaf $E$ of class $v$ is $\sigma_{t_0}^{H'}$-semistable for some $t_0$, then $E$ is $\sigma_t^{H'}$-semistable for all $t\geq t_0$.
\end{lemma}
\begin{proof}

First, notice that $E\in\mathcal{A}_t^{H'}$ for all $t\geq t_0$. If  $E$ is unstable for some $t>t_0$ then there is $t_1\geq t_0$ such that $E$ is $\sigma_t^{H'}$-semistable for all $t\in[t_0,t_1]$ and strictly semistable at $t_1$. Notice that if $0\rightarrow A\rightarrow E\rightarrow B\rightarrow 0$ is a short exact sequence in $\mathcal{A}_{t_1}^{H'}$ with $\langle ch(A), \alpha_{t_1}\rangle=0$ and $\langle ch(A), \alpha_{t}\rangle>0$ for all $t>t_1$, then $A$ can not be a subsheaf of $E$ since in such case we would have
$$
\langle ch(A),\alpha_t\rangle=r(A)\left( \chi(A\otimes L)-\frac{\chi(E\otimes L)}{r(E)}+ \left(\mu_{H}(A)-\mu_{H}(E)\right)t\right)\leq 0 \ \ \text{for}\ \ t>t_1,
$$ 
due to the $L$-twisted $H$-Gieseker semistability of $E$. In particular $A$ can not have rank 1.

 Assume that $\mathcal{H}^{-1}(B)\neq 0$ and let $t_2>t_1$ such that the first $H'$-semistable factor $F_1$ of $\mathcal{H}^{-1}(B)$ has slope $\mu_{H'}(F_1)=\frac{K\cdot H'}{2}-t_2(H')^2$. Then $\langle ch(F_1),\alpha_{t_2}\rangle<0$ and therefore the quotient $B/F_1[1]$ in $\mathcal{A}_{t_2}^{H'}$ satisfies $\langle ch(B/F_1[1]),\alpha_{t_2}\rangle<0$. Thus, if $\mathcal{K}$ denotes the kernel in $\mathcal{A}_{t_2}^{H'}$ of the map $E\rightarrow B/F_1[1]$ then $\langle ch(\mathcal{K}),\alpha_{t_2}\rangle>0$. Since $r(\mathcal{K})=r(A)-r(F_1)$ then by induction on the rank of destabilizing subobjects of $E$ for $t>t_0$, we know that $\mathcal{K}$ destabilizes $E$ for all $t\leq t_2$ as long as $\mathcal{K}$ is a subobject of $E$ in $\mathcal{A}_t^{H'}$, in particular $\mathcal{K}$ destabilizes $E$ at $t_1$. Thus $\mathcal{H}^{-1}(B)=0$ and therefore $A$ is a subsheaf of $E$ and can not destabilize $E$ for any $t>t_1$.
\end{proof}

Assume that $H$ and $H'$ are ample classes such that the set $\{sH+tH'\colon s,t>0\}$ is contained in a chamber for slope stability for the class $v$. We can consider the two-dimensional family of stability conditions $\sigma_{s,t}:=\sigma_{\alpha_{s,t}}^{H'}$ given by the vectors
$$
\alpha_{s,t}=(1,-\frac{K}{2}+sH+tH',d_{s,t}),
$$
where $d_{s,t}$ is chosen such that $\langle v,\alpha_{s,t}\rangle=0$. To ease the notation we denote the heart $\mathcal{A}^{H'}_{\alpha_{s,t}}$ by $\mathcal{A}_{s,t}$. Since our goal is to study the change of polarization for the Gieseker moduli and $M_H(v)\cong M_H(v\cdot ch(A))$ for any line bundle $A$, then by twisting the class $v$ if necessary we can assume that 
\begin{equation}\label{assumption}
\frac{K^2}{8} > \chi(\mathcal{O})-\frac{\chi(v)}{r(v)},\ \ \text{and}\ \ \mu_{H'}(v)>\frac{KH'}{2}. 
\end{equation}

This assumption is harmless since for instance it can be arranged by twisting by $nH$ for $n$ sufficiently large and divisible. This guarantees that $\sigma_{s,t}$ is a stability condition for all $s,t\geq 0$ and that every $H'$-semistable object of class $v$ is in the category $\mathcal{A}_{s,t}$. Under these assumptions we can describe the walls in the two-dimensional slice $\{\sigma_{s,t}\}_{s,t\geq 0}$.

\begin{corollary}\label{important2}
Every wall for the class $v$ in the quadrant $\{\sigma_{s,t}\}_{s,t\geq 0}$ of stability conditions is a line of nonpositive slope, it has slope zero (or infinity) if and only if $H$ (or $H'$) is on a wall in the wall and chamber decomposition of the ample cone for Gieseker stability with respect to the class $v$.  
\end{corollary}

\begin{proof}
We first describe the horizontal walls. Assume that there is an inclusion $0\rightarrow A\rightarrow E\rightarrow B\rightarrow 0$ of objects in $\mathcal{A}_{s_0,t_0}$ with $\mu_{H}(A)=\mu_H(E)$. If $\mathcal{H}^{-1}(B)\neq 0$ then there is $t_1>t_0$ such that $\mu_{H'}(F_1)=\frac{K_XH'}{2}-(s_0+t_1)(H')^2$, where $F_1$ is the first $H'$-semistable factor of $\mathcal{H}^{-1}(B)$, then the sheaf quotient $A/F_1$ is a destabilizing subobject of $E$ in $\mathcal{A}_{s_0,t_1}$ contradicting Lemma \ref{important}. Thus $A$ must be a subsheaf of $E$ implying that $H$ is on wall for slope semistability with respect to the class $v$. Moreover, since such wall is given by $t=c>0$ then $A$ should destabilize $E$ with respect $H$-Gieseker stability. If on the other hand 
\begin{equation}\label{onlyhere}
0\rightarrow A\rightarrow E\rightarrow B\rightarrow 0 \tag{$\diamond$}
\end{equation}
is a short exact sequence of sheaves destabilizing $E$ with respect to $H$-Gieseker stability then for 
$$
t_0=\left(\frac{\chi(A)}{r(A)}-\frac{\chi(E)}{r(E)}\right)\big/(\mu_{H'}(E)-\mu_{H'}(A))
$$
and $s$ large enough \eqref{onlyhere} is a short exact sequence in $\mathcal{A}_{s,t_0}$ and so produces a horizontal wall in the quadrant $\{\sigma_{s,t}\}_{s,t\geq 0}$.

Assume that $\mathcal{W}$ is a wall with positive slope destabilizing a sheaf $E$ of class $v$ that is both $H$-stable and $H'$-stable. We can assume that $E$ is destabilized in a small open set of $\mathcal{W}$ containing the point $(s_0,t_0)$. By Lemma \ref{important} $E$ is also $\sigma_{s_0+\epsilon,t_0}$-semistable and therefore $\sigma_{s_0+\epsilon,t_0+\delta}$-semistable for all $\delta>0$ contradicting that $\mathcal{W}$ is a wall destabilizing $E$ near $(s_0,t_0)$. The exact same argument shows that if $E$ is $H$-stable and strictly $H'$-semistable then the existence of a vertical wall will prevent the existence of any wall of positive slope destabilizing $E$.
\end{proof}

\begin{rem}\label{general}
We can consider the $n$-dimensional family of stability conditions $\sigma^{H}_{\mathfrak{a}}$ given by the vectors
$$
\alpha_{\mathfrak{a}}=(1,-\frac{K_X}{2}+a_1H_1+\cdots +a_nH_n, d_{\mathfrak{a}})
$$
where all the $H_i$'s are in the closure of the same chamber with respect to Gieseker stability for the class $v$, $a_k\geq 0$ for all $k$, $H$ is an interior point of the cone $\{a_1H_1+\cdots +a_nH_n\colon a_k\geq 0 \}\subset \mbox{Amp}(X)_{\Q}$, and again we choose $d_{\mathfrak{a}}$ such that $\langle v, \alpha_{\mathfrak{a}}\rangle=0$. Then Corollary \ref{important2} remains true: a wall destabilizing is a hyperplane that intersects each ``axis'' non negatively, it is parallel to the $H_i$-axis and lies in $\{\sigma_{\mathfrak{a}}\}_{a_k\geq 0}$ if and only if $H_i$ lies on a wall with respect to Gieseker stability for the class $v$. This can be easily proven by induction on $n$.  
\end{rem}

With the assumptions in \eqref{assumption} we can now prove Theorem \ref{main}.

\begin{proof}[Proof of Theorem \ref{main}]
We prove the case $L=0$, but the proof works for any $\Q$-line bundle $L$. Because of our assumptions and Theorem \ref{maciocia} we know that the walls on the ray $(0,t)$ are finite, therefore there are only finitely many walls intersecting this ray (since at a given point there are only finitely many Chern characters responsible for a wall). All these walls are lines of negative slope and therefore there exists a constant $S>0$ such that the walls coming from the ray $(0,t)$ intersect the ray $(s,0)$ at some $0<s<S$. 

Let $T$ be the upper bound for the walls on the ray $(0,t)$, i.e., for all $t>T$ the only $\sigma_{0,t}$-semistable objects of class $v$ are $H'$-Gieseker semistable sheaves (and therefore $H$-Gieseker semistable). We will prove that the walls in the ray $(s,0)$ are bounded by $M=\max\{S,T\}$. 

Assume that there is a destabilizing sequence 
\begin{equation}\label{exactsequence}
0\rightarrow A\rightarrow E\rightarrow B\rightarrow 0 \tag{$\star$}
\end{equation}
in $\mathcal{A}_{s_0,0}$, destabilizing an $H$-Gieseker semistable sheaf $E$ of class $v$ for some $s_0>M$. The corresponding wall $\mathcal{W}=\mathcal{W}(A,E)$ in the $(s,t)$-plane is a line of negative slope. Notice that $\mathcal{A}_{s_0,0}=\mathcal{A}_{0,s_0}$ and therefore \eqref{exactsequence} is also an exact sequence in $\mathcal{A}_{0,s_0}$. But since $s_0>T$ we know that $E$ is $\sigma_{0,s_0}$-stable and $A$ can not destabilize $E$ at this point. Thus we obtain 
\begin{align*}
\left(\frac{\chi(A)}{r(A)}-\frac{\chi(E)}{r(E)}\right)+s_0(\mu_{H'}(A)-\mu_{H'}(E)) &<0,\\
\left(\frac{\chi(A)}{r(A)}-\frac{\chi(E)}{r(E)}\right)+s_0(\mu_{H}(A)-\mu_{H}(E)) &=0.
\end{align*}

Therefore the slope of $\mathcal{W}$ is $>-1$ and $B\in\mathcal{A}_{s,t}$ for all $(s,t)\in\mathcal{W}$ ($s,t\geq 0$).

Let 
$$
0\rightarrow A_n\rightarrow A\rightarrow F_n\rightarrow 0
$$
be the last part of the Harder-Narasimhan filtration of $A$ with respect to $H'$-stability. Since $A$ can not destabilize $E$ at $(0,t_{\mathcal{W}})=\mathcal{W}\cap \{(0,t)\colon t\geq 0\}$ and $E,B\in \mathcal{A}_{0,t_{\mathcal{W}}}$, then $A$ can not belong to this category. Thus, there is $(s_1,t_1)\in\mathcal{W}$ such that $\mu_{H'}(F_n)=\frac{K_XH'}{2}-(s_1+t_1)(H')^2$ and therefore $\langle ch(A_n), \alpha_{s,t} \rangle>0$ for $(s,t)$ near $(s_1,t_1)$. This implies that the wall $\mathcal{W}'=\mathcal{W}(A_n,E)$ intersects $\mathcal{W}$. The slopes of $\mathcal{W}$ and $\mathcal{W}'$ must be the same, otherwise repeating the argument in the proof of Corollary \ref{important2} (moving right and then up) we will contradict that $\mathcal{W}$ is a wall. By finiteness of the Harder-Narasimhan filtration of $A$ we conclude that $\mathcal{W}$ should extend to intersect the ray $(0,t)$ contradicting the choice of $M$.
\end{proof}


\section{Determinant line bundles}
The choice of $\alpha_t\in v^{\perp}=\{w\colon \langle w,v\rangle=0\}$ not only does come naturally from the Gieseker condition, but it also produces (for large $t$) an ample line bundle on the coarse moduli space $M_H(v)$. 



Let $\mathcal{E}$ be a flat family of sheaves of topological type $v$ on $X$ parametrized by $S$. Denote by $[\mathcal{E}]$ its class in $K^0(X\times S)$, and the projections from $X\times S$ to $X$ and $S$ as in the diagram
$$
\begin{diagram}
\node{X\times S}\arrow{s,r}{p}\arrow{e,t}{q}\node{X}\\
\node{S}\node{}
\end{diagram}
$$                                                                                                                                                                                                  
Notice that $p$ is a smooth morphism, so $p_!: K^0(X\times S)\rightarrow K^0(S)$ is well defined. 
\begin{definition}\label{def3.1}
Define $\lambda_{\mathcal{E}}: K(X)\longrightarrow \Pic(S)$ as the composition of the homomorphisms:
$$
\begin{diagram}\dgARROWLENGTH=1.5em
\node{K(X)}\arrow{e,t}{q^*}\node{K^0(X\times S)}\arrow{e,t}{\cdot[\mathcal{E}]}\node{K^0(X\times S)}\arrow{e,t}{p_!}\node{K^0(S)}\arrow{e,t}{\mbox{det}}\node{\mbox{Pic}(S).}
\end{diagram}
$$
\end{definition}

 This is the ``Fourier--Mukai transform" with kernel $\mathcal{E}$, composed with the determinant homomorphism, which associates to a bounded complex of locally free sheaves $F^{\bullet}$ on $S$  its determinant line bundle.

If $L$ is a line bundle on $S$, it is easy to check that 
\begin{eqnarray}\label{universal}\lambda_{\mathcal{E}\otimes p^*L}(u)\cong\lambda_{\mathcal{E}}(u)\otimes L^{\chi(u\cdot v)}.
\end{eqnarray} 

Assume for the moment that there is a universal sheaf $\mathcal{E}$ on $X\times M^{s}_H(v)$. Such universal sheaf would only be defined up to tensoring with the pull back of a line bundle from the base. If we choose $u$ such that $\chi(u\cdot v)=0$, i.e., $ch(u)\cdot \mbox{td}(X)\in v^{\perp}$, then by (\ref{universal}), $\lambda_{\mathcal{E}}(u)$ would not depend on the ambiguity of the choice of the universal sheaf and therefore would yield a line bundle on $M^s_H(v)$. We will simply write $\lambda(u)$ for this line bundle on $M^s_H(v)$. 

In general, there is no universal sheaf on the coarse moduli space $M_H(v)$. The determinant line bundle is just a line bundle on the moduli stack. This line bundle, however, always descends to $M^s_H(v)$ for $ch(u)\cdot\mbox{td}(X)\in v^{\perp}$. If $M^s_H(v)\subsetneq M_H(v)$, one requires $ch(u)\cdot\mbox{td}(X)\in v^{\perp}\cap \{1,[\mathcal{O}_H]\}^{\perp\perp}$, to guarantee that this line bundle does descend to the whole coarse moduli space $M_H(v)$ (cf., \cite[Theorem 8.1.5]{HL}). 

There are two distinguished determinant line bundles on $M_H(v)$ given by taking
\begin{eqnarray}\label{generator}u_i=-c_0\cdot h^i+\chi(v\cdot h^i)[\bb{C}_x],\  i=0,1,
\end{eqnarray}
where $h=[\mathcal{O}_H]\in K(X)$, and $x\in X$. The line bundle $\mathcal{L}_i:=\lambda(u_i)$ does not depend on the choice of the point $x$. The importance of the line bundles $\mathcal{L}_0$ and $\mathcal{L}_1$ is represented in the following theorem. 

\begin{theorem}{\cite[Theorem 8.1.11]{HL}} The line bundle $\mathcal{L}_0\otimes \mathcal{L}_1^{\otimes m}\in \mbox{Pic}(M_H(v))$ is ample for $m\gg 0$.
\end{theorem}

An easy computation shows
\begin{align*}
(ch(u_0)+tch(u_1))\cdot \mbox{td}(X)&=(-c_0,-c_0(tH),c_1(tH)-\frac{c_0}{2}K_X(tH)+\chi(v))\cdot \mbox{td}(X)\\
&=-c_0(1,-\frac{K_X}{2}+tH,d_t)\\
&=-c_0\alpha_t.
\end{align*}

Thus the determinant line bundle $\mathcal{L}_t:=\lambda(-\alpha_t)\in\mbox{Pic}(M_H(v))_{\R}$ is ample for $t\gg 0$. In fact, it will turn out that $\mathcal{L}_t$ is ample as long as $\sigma_t$-semistability and $H$-Gieseker semistability coincide. 

In \cite{BM}, Bayer and Macr\'i associate to every numerical stability condition $\sigma=(Z,\mathcal{A})\in\mbox{Stab}(X)$ a numerical class $\mathcal{L}_{\sigma}\in N^1(\mathcal{M}_{\sigma})$ on the moduli stack of $\sigma$-semistable objects. If $C$ is an integral projective curve in $\mathcal{M}_{\sigma}(v)$ corresponding to a family $\mathcal{E}\in D^b(C\times X)$, then the class $\mathcal{L}_{\sigma}$ is defined via the intersection rule
$$
\mathcal{L}_{\sigma}\cdot C=\mathfrak{Im}\left(-\frac{Z(\Phi_{\mathcal{E}}(\mathcal{O}_C))}{Z(v)}\right),
$$  
where $\Phi_{\mathcal{E}}\colon D^b(C)\rightarrow D^b(X)$ denotes Fourier--Mukai transform with kernel $\mathcal{E}$. Their main result is:
\begin{lemma}[Positivity Lemma, {\cite[Lemma 3.3]{BM}}] The divisor class $\mathcal{L}_{\sigma}$ is nef: $\mathcal{L}_{\sigma}\cdot C\geq 0$. Further, we have $\mathcal{L}_{\sigma}\cdot C>0$ if and only if for two general closed points $c,c'\in C$, the corresponding objects $\mathcal{E}_c,\mathcal{E}_{c'}\in D^b(X)$ are not $S$-equivalent.
\end{lemma}

\begin{rem}\label{ample_line_bundles} In the case of the stability conditions $\sigma_t$ corresponding to the vectors $\alpha_t$, we obtain
$$
\mathcal{L}_{\sigma_t}\cdot C=-\frac{1}{\langle v, \alpha_tH'\rangle}\langle \Phi_{\mathcal{E}}(\mathcal{O}_C), \alpha_t\rangle.
$$
On the other hand, by Theorem \ref{main} there exists $t_0$ such that $\sigma_t$-semistability equals $H$-Gieseker semistability for $t>t_0$. Assume that $H$ is not on a wall for Gieseker semistability, then the moduli space $M_H(v)$ consists only of stable sheaves. In this case, the class $\mathcal{L}_{\sigma}$ defines a nef line bundle on the Gieseker moduli space $M_H(v)$ (\cite[Remark 2.6]{BM}). By \cite[Proposition 4.4]{BM} we obtain
$$
\mathcal{L}_{\sigma_t}\equiv \frac{1}{\langle v, \alpha_tH'\rangle}\lambda(-\alpha_t)\in \mbox{Nef}(M_H(v))\ \ \text{for}\ \ t>t_0.
$$
Thus $\mathcal{L}_{\sigma_t}$ and $\mathcal{L}_t$ are on the same ray in $N^1(M_H(v))$. Since the family of line bundles $\{\mathcal{L}_t\}_t$ lie on a line in $N^1(M_H(v))$, and $\mathcal{L}_t$ is ample for $t$ sufficiently large, then the convexity of $\mbox{Nef}(M_H(v))$ implies that $\mathcal{L}_t$ is ample for $t>t_0$.
\end{rem}

\section{Change of the polarization}\label{changeofpolarization}

We will extend Theorem \ref{main} just a little, enough to prove the following

\begin{theorem}\label{main2}
Let $H^+$, $H^-$ be two ample divisors on adjacent chambers with respect to the Gieseker wall and chamber decomposition of $\mbox{Amp}(X)_{\Q}$ for the class $v$. There is a one-parameter family of stability conditions $\{\gamma_t\}_{t\in I}$ and $t_0<t_1<\cdots <t_n$ in $I$ such that the moduli spaces $M_{\gamma_t}(v)$ are isomorphic for all $t\in(t_i,t_{i+1})$. Moreover, each of these moduli spaces is isomorphic to a moduli space of twisted sheaves. For $t<t_0$, $M_{\gamma_t}(v)$ coincides with the moduli space of $H^+$-Gieseker semistable sheaves, and for $t>t_n$ it coincides with the moduli of $H^-$-Gieseker semistable sheaves.
\end{theorem}

\begin{proof}
Let $H_0$ be an ample class on a wall separating the chambers containing $H^+$ and $H^-$. It is enough to find a family of stability conditions reflecting the change of polarization from $H^+$ to $H_0$. Consider the two-dimensional family of stability conditions $\sigma^{H^+}_{s,t}$ given by the vectors $\alpha_{s,t}=(1,-\frac{K}{2}+sH_0+tH^+,d_{s,t})$. By Corollary \ref{important2} we know that in the quadrant $\{\sigma^{H^+}_{s,t}\}_{s,t\geq 0}$ there are horizontal walls (finitely many since there are only finitely many Chern characters of subsheaves destabilizing a Gieseker semistable sheaf \cite[Proposition 1.6]{MW}). The key point in the proof of Theorem \ref{main} that allows us to get only finitely many walls is the absence of horizontal walls, so in the first quadrant above the first horizontal wall there are only finitely many walls and the same remains true between two consecutive horizontal walls. Then by choosing $s_0$ large enough we know that the walls intersecting the ray $\Lambda_{s_0}=\{\sigma^{H^+}_{s_0,t}\colon t\geq 0\}$ are all horizontal. 

Note that if $t$ is a positive rational number then an object is $\sigma^{H^+}_{s_0,t}$-semistable if and only if it is a $tH^+$-twisted $H_0$-Gieseker semistable sheaf, in particular $\sigma^{H^+}_{s_0,t}$-semistable objects have projective moduli \cite{MW}. 

By Corollary \ref{important2} we know that a wall on the ray $\Lambda_{s_0}$ is produced if there is an $H_0$-semistable sheaf of class $v$ that is $H^+$-Gieseker semistable but that fails to be $H_0$-Gieseker semistable, i.e., if there is an inclusion of torsion-free sheaves $A\hookrightarrow E$ with $ch(E)=v$ such that
\begin{align*}
\mu_{H_0}(A)&=\mu_{H_0}(v),\\
\mu_{H^+}(A)&<\mu_{H^+}(v),\  \text{and}\\
\frac{\chi(A)}{r(A)}&>\frac{\chi(v)}{r(v)}.
\end{align*}

This inclusion will produce the wall 
$$
t=-\left(\frac{\chi(A)}{r(A)}-\frac{\chi(v)}{r(v)}\right)\big/ \left(\mu_{H^+}(A)-\mu_{H^+}(v)\right)
$$
which is rational. Then there are $t_0<t_1<\cdots< t_k$ rational numbers corresponding to walls on the ray $\Lambda_{s_0}$ such that for $t>t_k$, $\sigma^{H^+}_{s_0,t}$-semistability coincides with $H^+$-Gieseker semistability.  
\end{proof}

\begin{rem} If we consider the stability conditions of Remark \ref{general}, then the proof of Theorem \ref{main2} and Remark \ref{ample_line_bundles} guarantee the existence of a convex and polyhedral chamber $\mathcal{C}\subset \Delta=\{a_1H_1+\cdots +a_nH_n\colon a_i\geq 0\}$ such that the determinant line bundle associated to a stability condition for a polarization in $\mathcal{C}$ is an ample line bundle on the Gieseker moduli.  
\end{rem}
\begin{figure}[!ht]
\centering
\includegraphics[scale=0.9]{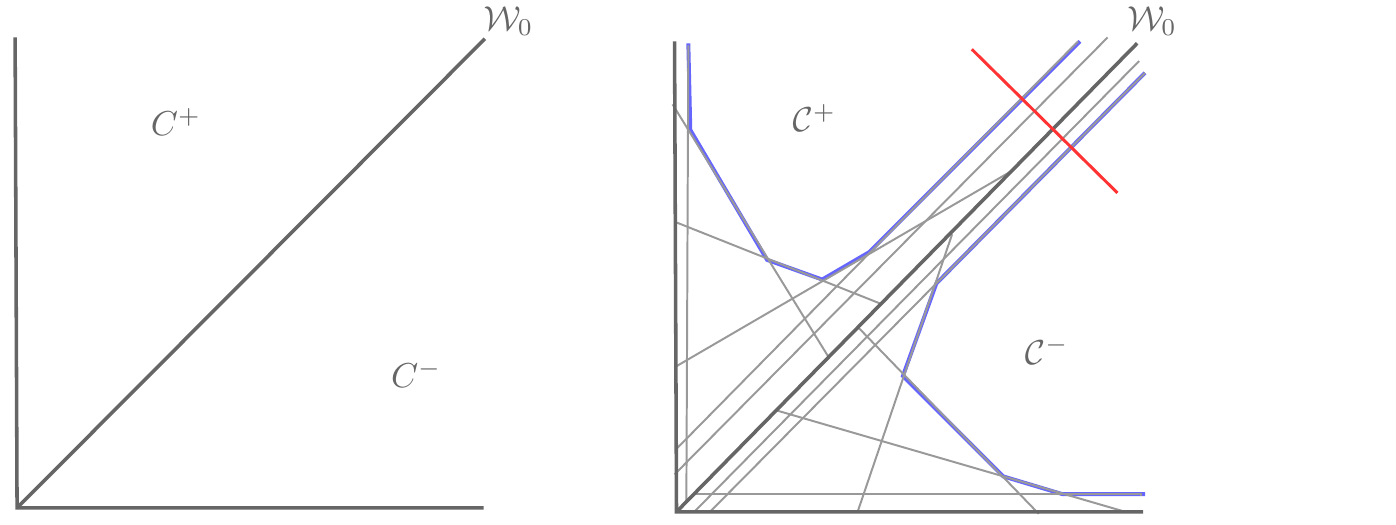}
\caption{Gieseker and Bridgeland chambers.}\label{chambers}
\end{figure}

\begin{rem}\label{chambersexplained}
 Figure \ref{chambers} shows the typical picture of walls for slope stability and our Bridgeland walls. The red line corresponds to the one-dimensional family of stability conditions in Theorem \ref{main2}. The chambers $\mathcal{C}^{\pm}$ are slices of $\mbox{Amp}(M_{C^{\pm}}(v))$. Notice that $\mathcal{W}_0$ may or may not be a Bridgeland wall, it will be if and only if there is an inclusion of sheaves $A\hookrightarrow E$ of the same reduced Hilbert polynomial and with $E$ being $C^+$ or $C^-$-stable of class $v$. If it happens that every $H_0$-Gieseker semistable sheaf is $H_0$-Gieseker stable for $H_0\in \mathcal{W}_0$ then $\mathcal{W}_0$ is not a Bridgeland wall and there are no Bridgeland walls parallel to $\mathcal{W}_0$, thus $\mathcal{C}^+$ and $\mathcal{C}^-$ give a single chamber. If every $C^{\pm}$-semistable sheaf is also $H_0$-Gieseker semistable and some $C^{\pm}$-semistable sheaf is strictly $H_0$-Gieseker semistable then $\mathcal{W}_0$ is a Bridgeland wall and there are no Bridgeland walls parallel to $\mathcal{W}_0$ on the side of $C^{\pm}$.   
\end{rem}

\begin{example}\label{p1p1example}
Let $X=\P^1\times\P^1$ and $v=(2,0,-5)$. The advantage of studying rank-2 sheaves is that computing the walls for slope stability is very simple, but even in this case we can see some of the phenomena described in Remark \ref{chambersexplained} already happening. A wall for slope stability is produced by a short exact sequence
$$
0\rightarrow L\rightarrow E \rightarrow \mathcal{I}_Z(L^{\vee})\rightarrow 0
$$
for some line bundle $L$ and a zero-dimensional subscheme $Z\subset X$ satisfying
$$
L^2+5=\ell(Z).
$$ 

If $H_1=\mathcal{O}(1,0)$ and $H_2=\mathcal{O}(0,1)$, then $L$ should also satisfy that for some integers $a,b\geq 0$
$$
L\cdot (aH_1+bH_2)=0.
$$

\begin{figure}[!ht]
\centering
\includegraphics[scale=0.9]{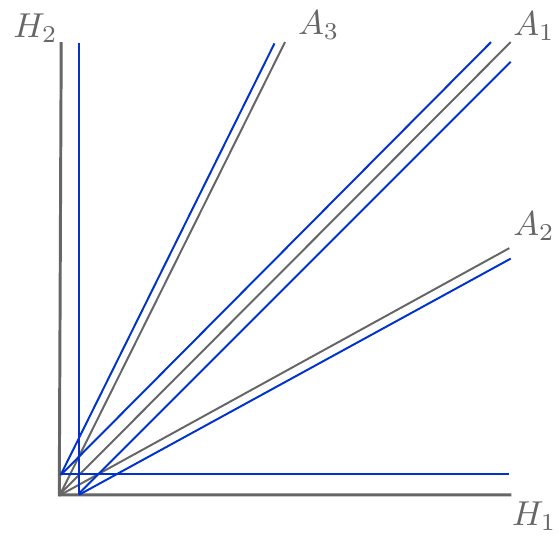}
\caption{Parallel walls for $\P^1\times \P^1$ and $v=(2,0,-5)$.}\label{chambers2}
\end{figure}

The ray generated by $aH_1+bH_2$ is a wall. Thus $L=\pm (aH_1-bH_2)$ and $L^2=-2ab$ which implies $L^2=-4,\ -2$, or $0$. Therefore the walls for slope stability with respect to the class $v$ are given by the polarizations
$A_1=H_1+H_2$, $A_2=2H_1+H_2$, $A_3=H_1+2H_2$, $H_1$, and $H_2$. Since $\chi(E)/r(E)=-3/2$ and $\chi(L)$ is an integer then none of these walls is a Gieseker wall. However, each of the walls for slope stability produces a Bridgeland wall as in Figure \ref{chambers2}. 

\end{example}
\section{Change of polarization for 1-dimensional sheaves}\label{change_for_1-dim}

Gieseker semistability for 1-dimensional sheaves also depends on the polarization but its variation is somehow less studied. As we have seen, in the positive rank case if $H^+$ and $H^-$ are polarizations in adjacent chambers, then there is a rational map
$$
\xymatrix{
M_{H^+}(v)\ar@{-->}[r]& M_{H^-}(v),}
$$
which factors through a finite sequence of Thaddeus flips involving moduli spaces of twisted sheaves. The difference in the 1-dimensional case lies in the fact that the Hilbert polynomial is linear and therefore if $H_0$ lies on a wall separating the chambers containing $H^+$ and $H^-$, respectively, then we obtain a diagram 
\begin{equation}\label{1_dimensional_gieseker_wall_crossing}
\xymatrix{
M_{H^+}(v)\ar@{-->}[rr]\ar[dr] & & M_{H^-}(v),\ar[dl]\\
& M_{H_0}(v) &
}
\end{equation}
where the morphisms $M_{H^{\pm}}(v)\rightarrow M_{H_0}(v)$ send a 1-dimensional $H^{\pm}$-Gieseker semistable sheaf to its $S$-equivalence class with respect to $H_0$-Gieseker semistability. The local behavior of the morphisms $M_{H^{\pm}}(v)\rightarrow M_{H_0}(v)$ at a strictly $H_0$-Gieseker semistable sheaf has been recently studied by Arbarello and Sacc\`a \cite{AS15}, in the case of K3 surfaces. 

In this section, we produce a family of stability conditions realizing the change of polarization for 1-dimensional sheaves, i.e., we obtain all the diagrams \eqref{1_dimensional_gieseker_wall_crossing} as Bridgeland wall-crossings. In fact, for our family of stability conditions there is a one-to-one correspondence between Bridgeland walls and Gieseker walls. 

As always, fix an ample class $H\in \mbox{Amp}(X)$.

\begin{definition}
A pure 1-dimensional torsion sheaf $\mathcal{F}$ is $H$-Gieseker (semi)stable if for all subsheaves $A\hookrightarrow \mathcal{F}$ one has
\begin{equation}\label{1_dim_gieseker_stab}
\frac{\chi(A)}{c_1(A)\cdot H}\semi \frac{\chi(\mathcal{F})}{c_1(\mathcal{F})\cdot H}.
\end{equation}
\end{definition}

If we try to write the inequality \eqref{1_dim_gieseker_stab} as a condition on the sign of the pairing of $ch(A)$ with a class $\beta\in \widetilde{\mbox{NS}}(X)$, depending on the class $v=ch(\mathcal{F})$, we obtain the equivalent 
\begin{definition}
A pure 1-dimensional torsion sheaf $\mathcal{F}$ of class $v=(0,C,\frac{KC}{2}+\chi)$ is $H$-Gieseker semistable if for all subsheaves $A\hookrightarrow \mathcal{F}$ one has
\begin{equation}
\langle ch(A), \beta^{H}_{\tau} \rangle \leq 0,
\end{equation}
where 
$$
\beta^H_{\tau}=\left(1,-\frac{K_X}{2}-\frac{\chi}{CH}H,-\tau\right).
$$
\end{definition} 
We denote the moduli space parametrizing $S$-equivalence classes of 1-dimensional $H$-Gieseker semistable sheaves of class $v$ by $N_H([C],\chi)$.
  
Notice that since every subsheaf of a pure 1-dimensional sheaf is again pure and 1-dimensional, then any choice of $\tau$ will give the inequality \eqref{1_dim_gieseker_stab}. However, we will require $\tau$ to be positive and large enough in order for $\beta^H_{\tau}$ to satisfy the Bogomolov inequality and give a stability condition. 

\begin{prop}\label{walls_1_dim}
For $\tau>0$ and large enough, and any ample class $H'\in\mbox{Amp}(X)$, the vector $\beta^H_{\tau}$ gives a stability condition $\sigma^{H'}_{\tau}\in\mbox{Stab}(X)$. For $\tau\gg 0$ the only $\sigma^{H'}_{\tau}$-semistable objects of class $v$ are the 1-dimensional $H$-Gieseker semistable sheaves. Moreover, the walls in $\mbox{Stab}(X)$ of type $v$ intersecting the ray $\{\sigma^{H'}_{\tau}\}_{\tau>0}$ are bounded above.
\end{prop}
\begin{proof}
If we choose $\tau$ such that 
$$
\left(-\frac{K_X}{2}-\frac{\chi}{CH}H\right)^2>-2\tau,
$$
then \eqref{central_charge} is the charge of a stability condition for any $H'\in\mbox{Amp}(X)$. The boundedness result follows from Theorem \ref{maciocia} since the degree 1 component of $\beta^H_{\tau}$ is constant. 

To see that in the last chamber we obtain only the 1-dimensional $H$-Gieseker semistable sheaves notice that if $A\hookrightarrow E$ is an inclusion in $\mathcal{A}^{H'}_{\beta^H_{\tau}}$  with $ch(E)=v$ and $ch_0(A)<0$, then for $\tau$ sufficiently large 
$$
\langle ch(A), \beta^H_{\tau}\rangle >0.
$$  

This forces any $\sigma^{H'}_{\tau}$-semistable object $E$ in the last chamber ($\tau\gg 0$) to be a sheaf since otherwise $\mathcal{H}^{-1}(E)[1]$ would destabilize $E$. Since $\beta^H_{\tau}$ was obtained directly from inequality \eqref{1_dim_gieseker_stab} then $E$ is $H$-Gieseker semistable. On the other hand, if $\mathcal{F}$ is a 1-dimensional $H$-Gieseker semistable sheaf with $ch(\mathcal{F})=v$ then $\mathcal{F}\in\mathcal{A}^{H'}_{\beta^H_\tau}$ and it is $\sigma^{H'}_{\tau}$-semistable for $\tau$ in the last chamber since for such $\tau$ no subobject of positive rank can destabilize $\mathcal{F}$. 
\end{proof}
It is convenient to have a version of Lemma \ref{important}.

\begin{lemma}\label{important_1_dim}
If $\mathcal{F}$ is a 1-dimensional $H$-Gieseker semistable sheaf with $ch(\mathcal{F})=v$ that is $\sigma^{H'}_{\tau_0}$-semistable for some $\tau_0>0$, then $\mathcal{F}$ is $\sigma^{H'}_{\tau}$-semistable for each $\tau\geq \tau_0$.  
\end{lemma}

\begin{proof}
First, notice that the category $\mathcal{A}^{H'}_{\tau}$ does not depend on $\tau$. If $A\hookrightarrow \mathcal{F}$ is an inclusion in $\mathcal{A}^{H'}$ destabilizing $\mathcal{F}$ for some $\tau_1>\tau_0$ then $A$ can not be a subsheaf and so 
$$
\langle ch(A), \beta^H_{\tau}\rangle>0
$$ 
for every $\tau<\tau_1$, in particular $A$ would destabilize $E$ at $\tau_0$, a contradiction.
\end{proof}

If we want to study change of polarization for 1-dimensional sheaves, we need to introduce stability conditions that take into account Gieseker slopes for different values of $H$. At first glance, it seems like using the stability conditions $\sigma^{H'}_{\tau}$ is not the right choice because of the term $CH$ appearing as a denominator in $\beta^H_{\tau}$. However, a change of coordinates will allow us to stack all these stability conditions together as $H$ moves in the ample cone of $X$.

Since the functor $\mathcal{F}\mapsto \mathscr{E}xt^1(\mathcal{F}, \omega_X)$ induces an isomorphism between $N_{H}([C],\chi)$ and $N_{H}([C],-\chi)$ (see \cite[Corollary 3.3]{Mar}), then we can assume $\chi\geq 0$. Moreover, since $H$-Gieseker semistability for 1-dimensional sheaves of Euler-Poincar\'e characteristic $\chi=0$ is independent of the polarization, then we can assume $\chi>0$. 

Let $H,H'\in \mbox{Amp}(X)$ and assume for simplicity that $CH=CH'$. Let $D=s_0H'+t_0H$ for some $s_0,t_0\geq 0$ with $(s_0,t_0)\neq (0,0)$. Using the change of coordinates   
$$
\tau=-\frac{K_X^2}{8}-\frac{\chi}{2}\frac{DK_X}{DC}+\frac{\chi}{\lambda DC}
$$ 
we obtain the vectors
\begin{align*}
\alpha_{\lambda D} :=&\frac{(\lambda D)C}{\chi}\beta^D_{\tau}\\
=& \left(\frac{\lambda DC}{\chi}, -\frac{\lambda D C}{\chi}\frac{K_X}{2}-\lambda D, -\frac{\lambda DC}{\chi}\frac{K_X^2}{8}+\frac{\lambda DK}{2}-1\right).
\end{align*}

We can then consider $s$ and $t$ as parameters for the stability conditions given by the vectors $\alpha_{s,t}:=\alpha_{D_{s,t}}$ with $D_{s,t}=sH'+tH$. A simple computation shows that $\alpha_{s,t}$ satisfies the Bogomolov inequality for every $(s,t)\in \R_{\geq 0}\times \R_{\geq 0}\setminus \{(0,0)\}$. 

Now, since $\lambda\rightarrow 0$ as $\tau \rightarrow +\infty$ then we know that along every line $\{(s,t)=\lambda(s_0,t_0)\}_{\lambda>0}$ with $(s_0,t_0)\in\R_{\geq 0}\times \R_{\geq 0}\setminus \{(0,0)\}$, the only $\sigma^{H'}_{s,t}$-semistable objects of class $v$  for $0<\lambda \ll 1$ are the 1-dimensional  $D_{s_0,t_0}$-Gieseker semistable sheaves. Moreover, by Proposition \ref{walls_1_dim} along every such line the walls are bounded below.

To obtain a region where the change of polarization is realized, we need to study the configuration of walls in $\R_{\geq 0}\times \R_{\geq 0}$.

\begin{theorem}\label{main_1_dim}
If an inclusion $A\hookrightarrow E$ in $\mathcal{A}^{H'}_{s_0,t_0}$ produces a wall destabilizing a 1-dimensional sheaf $E$ of class $v$, then $A$ destabilizes $E$ for all $\sigma^{H'}_{s,t}$ with
$$
(0,0)\neq (s,t)\in\mathcal{W}_{ch(A)}=\{(s,t)\in \R_{\geq 0}\times \R_{\geq 0}\colon \langle ch(A), \alpha_{s,t} \rangle=0\}.
$$
Moreover, if $\mathcal{W}_{ch(A)}$ is a line passing through the origin, then $A$ must be a subsheaf destabilizing $E$ with respect to Gieseker semistability. 
\end{theorem}
\begin{proof}
The argument is very similar to the one used in the proof of Corollary \ref{important2}. Suppose that 
$$
0\rightarrow A\rightarrow E\rightarrow B\rightarrow 0
$$
is a short exact sequence in $\mathcal{A}^{H'}_{s_0,t_0}$ producing the wall $\mathcal{W}$. This sequence should remain exact for all $(s,t)\in\mathcal{W}$ in a neighborhood of $(s_0,t_0)$, the only way this is not the case is if the first $H'$-semistable factor of $\mathcal{H}^{-1}(B)$ pairs to zero with $\alpha_{s_0,t_0}H'$, which violates the semistability of $B$ at $(s_0,t_0)$. 

If this sequence is not exact for all $(s,t)\in\mathcal{W}$, then $\mathcal{W}$ should intersect another wall for $E$ that is a line of different slope in $\R_{\geq 0}\times \R_{\geq 0}$. This contradicts Lemma \ref{important_1_dim}. 

Now, if $ch(A)=(r',c',ch'_2)$ then the equation for $\mathcal{W}$ is
\begin{align*}
&s\left(\chi(A)-r'\left(\chi(\mathcal{O}_X)+\frac{K_X^2}{8}\right)-\frac{\chi}{CH}c'H'\right)\\
+& t\left(\chi(A)-r'\left(\chi(\mathcal{O}_X)+\frac{K_X^2}{8}\right)-\frac{\chi}{CH}c'H\right)-r'\frac{\chi}{CH}=0,
\end{align*}
which shows that $\mathcal{W}$ passes through $(0,0)$ if and only if $A$ is a subsheaf of $E$. If that is the case, the equation for $\mathcal{W}$ becomes
$$
c'(sH'+tH)\left(\frac{\chi(A)}{c'(sH'+tH)}-\frac{\chi}{C(sH'+tH)}\right)=0,
$$
which means that $sH'+tH$ is on a wall for Gieseker semistability for the class $v$.
\end{proof}
\begin{corollary}
Let $H,\ H'\in\mbox{Amp}(X)$ and fix a class $v=(0,C,\frac{KC}{2}+\chi)\in\widetilde{\mbox{NS}}(X)$ with $\chi>0$. Then there is a one-parameter family of stability conditions $\{\sigma_t\}_{t\in I}$ and $t_0<t_1<\cdots <t_n$ in $I$ such that the moduli spaces $M_{\sigma_t}(v)$ are isomorphic for all $t\in(t_i,t_{i+1})$, each of these moduli spaces is isomorphic to a moduli space of 1-dimensional Gieseker semistable sheaves for some polarization in the segment $S=\{aH'+(1-a)H\colon a\in[0,1]\}$. For $t<t_0$ $M_{\sigma_t}(v)$ coincides with the moduli space $N_{H}([C],\chi)$, and for $t>t_n$ it coincides with $N_{H'}([C],\chi)$. Moreover, there is a one-to-one correspondence between Bridgeland walls for the class $v$ intersecting the segment $\{\sigma_t\}_{t\in I}$ and walls for Gieseker semistability intersecting the segment $S$. 
\end{corollary}

\begin{proof} By Proposition \ref{walls_1_dim}, the walls intersecting the rays $\{\sigma^{H'}_{0,t}\}_{t>0}$ and $\{\sigma^{H'}_{s,0}\}_{s>0}$ are bounded below by a positive number $R$ not on a wall. By Theorem \ref{main_1_dim} there exists $0<\epsilon\ll 1$ such that the only Bridgeland walls for the class $v$ intersecting the segment $I=\{(1-t)(0,R)+t(R,0)\colon t\in (-\epsilon,1+\epsilon)\}$ are those walls passing through $(0,0)$. These walls are in one-to-one correspondence with Gieseker walls due to Theorem \ref{main_1_dim}.
\end{proof}

\section{Birational geometry of complex surfaces}\label{birational}

In this section we present a new approach to a result of Toda \cite{TodaPS} within the set of ideas surrounding the previous sections. The precise statement is
\begin{theorem}{\cite[Corollary 1.4]{TodaPS}}
Let $X$ be a smooth projective complex surface and let $\pi \colon X\rightarrow Y$ be the blow down of a $-1$-curve $C\subset X$. Then there is a continuous one parameter family of Bridgeland stability conditions $\{\sigma_t\}_{t\in (-1,1)}$ on $D^b\mbox{Coh}(X)$  such that $M_{\sigma_t}([\mathcal{I}_p])$ is isomorphic to $X$ for $t>0$ and isomorphic to $Y$ for $t<0$.
\end{theorem}
\begin{proof}
Choose a sufficiently ample line bundle $L$ on $Y$ such that $\pi^*(L)=D=H+C$ for some $H\in\mbox{Amp}(X)$. Consider the two-dimensional family of stability conditions $\sigma^{H}_{s,t}$ on $X$ given by the vectors
$\alpha_{s,t}=(1,-\frac{K_X}{2}+tH+sD,1)$. By choosing $s_0$ large enough we can assume
$$
-CH> \frac{K_XH}{2}-s_0DH, \ \ \text{and}\ \ \left(-\frac{K_X}{2}+tH+s_0D\right)^2>2 \ \ \ \text{for all}\ \ \ t\geq 0.
$$
Thus $0\rightarrow \mathcal{O}(-C)\rightarrow \mathcal{I}_p\rightarrow \mathcal{O}_C(-1)\rightarrow 0$ is a short exact sequence in $\mathcal{A}^H_{s,t}$ for every $s>s_0$ and $t\geq 0$. Moreover,
$$
\langle ch(\mathcal{O}(-C), \alpha_{s,0})\rangle=-sDC=0=\langle ch(\mathcal{I}_p),\alpha_{s,0}\rangle.
$$
Thus the ray $\sigma^H_{s,0}$ is a wall destabilizing all ideal sheaves $\mathcal{I}_p$ for $p\in C$. Now, if there were a horizontal wall in the quadrant $K=\{\sigma^H_{s,t}\colon s\geq s_0,\ t>0\}$ then it would have to be produced by a subsheaf of $\mathcal{I}_p$, such destabilizing object would have to be of the form $\mathcal{I}_Z(-C')$ for some zero-dimensional subscheme $Z\subset X$ containing $p$ and some curve $C'\subset X$, but this is not possible since being the wall horizontal will force $C'=C$ and $\chi(\mathcal{I}_Z(-C))-\chi(\mathcal{I}_p)=-\ell(Z)>0$.

 Therefore, by the proof of Theorem \ref{main} we know that there are only finitely many walls in the closure $\bar{K}$. Thus, we can choose $s_1>s_0$ such that there are no walls on the ray $\{\sigma^H_{s_1,t}\}_{t\geq 0}$ other than $\sigma^H_{s_0,0}$.
 
The moduli spaces $M_t([\mathcal{I}_p])$ for the stability conditions $\sigma^H_{s_1,t}$ with $t>0$ are all isomorphic to the moduli space of $s_1D$-twisted $H$-semistable sheaves of class $[\mathcal{I}_p]$, i.e., they are all isomorphic to $X$. Since
$$
\mbox{ext}^1(\mathcal{O}(-C),\mathcal{O}_C(-1))=h^1(\P^1,\mathcal{O}_{\P^1}(-2))=1,
$$ 
then $M_t([\mathcal{I}_p])$ for $-1\ll t\leq 0$ is naturally isomorphic to $Y$ and the map $M_{t}([\mathcal{I}_p])\rightarrow M_0([\mathcal{I}_p])$ coincides with the contraction $\pi$.
\end{proof}
\bibliographystyle{hep}
\bibliography{references.bib}

\end{document}